\documentclass[12pt]{amsart}
\usepackage{amsfonts, amssymb, latexsym, graphicx, hyperref, amsmath}
\usepackage[vcentermath]{youngtab}
\usepackage{tikz}
\usepackage{tikz-cd}
\usetikzlibrary{calc, shapes, backgrounds,arrows,positioning,decorations.pathmorphing}

\newtheorem{theorem}{Theorem}
\newtheorem{proposition}[theorem]{Proposition}

\newtheorem{lemma}[theorem]{Lemma}

\newtheorem{claim}[theorem]{Claim}
\newtheorem*{claim*}{Claim}
\newtheorem{corollary}[theorem]{Corollary}
\newtheorem{Main Conjecture}[theorem]{Main Conjecture}

\theoremstyle{remark}

\theoremstyle{plain}

\newcommand{\cellsize}{14}
\newlength{\cellsz} 
\newsavebox{\cell}
\sbox{\cell}{\begin{picture}(\cellsize,\cellsize)
\put(0,0){\line(1,0){\cellsize}}
\put(0,0){\line(0,1){\cellsize}}
\put(\cellsize,0){\line(0,1){\cellsize}}
\put(0,\cellsize){\line(1,0){\cellsize}}
\end{picture}}
\newcommand\cellify[1]{\def\thearg{#1}\def\nothing{}%
\ifx\thearg\nothing
\vrule width0pt height\cellsz depth0pt\else
\hbox to 0pt{\usebox{\cell} \hss}\fi%
\vbox to \cellsz{
\vss
\hbox to \cellsz{\hss$#1$\hss}
\vss}}
\newcommand\tableau[1]{\vtop{\let\\\cr
\baselineskip -16000pt \lineskiplimit 16000pt \lineskip 0pt
\ialign{&\cellify{##}\cr#1\crcr}}}

\newcommand{\excise}[1]{}

\begin{document}
\pagestyle{plain}
\title{Maximizing the Edelman-Greene Statistic}
\author{Gidon Orelowitz}
\address{Dept.~of Mathematics, U.~Illinois at Urbana-Champaign, Urbana, IL 61801, USA} 
\email{gidono2@illinois.edu}
\date{\today}

\begin{abstract}
The \emph{Edelman-Greene statistic} of S. Billey-B. Pawlowski measures the "shortness" of the Schur expansion of a Stanley symmetric function.
We show that the maximum value of this statistic on permutations of Coxeter length $n$ is the number of involutions in the symmetric group $S_n$, and explicitly describe the permutations that attain this maximum.
Our proof confirms a recent conjecture of C. Monical, B. Pankow, and A. Yong: we give an explicit combinatorial injection between 
a certain collections of  Edelman-Greene tableaux and standard Young tableaux.
\end{abstract}
\maketitle

\section{Introduction}
Let $S_n$ be the \emph{symmetric group} on $[n] = \{1,2,\dots, n\}$.  $S_n$ can be embedded in $S_{n+1}$ by the natural inclusion, and from this define $S_\infty = \bigcup_{n=1}^\infty S_n$.  Let $s_i\in S_{\infty}$ be the \emph{simple transposition} swapping $i$ and $i+1$.  Each $w\in S_\infty$ is expressible 
as a product of simple transpositions; the minimum possible length of such an expression is the \emph{Coxeter length} $\ell(w)$. An expression of length $\ell(w)$ is a  
\emph{reduced word} of $w$. Let ${\sf Red}(w)$ be the set of reduced words of $w$.  A permutation $w$ is \emph{totally commutative} \footnote{This is stricter than the definition of the similar sounding \emph{fully commutative} \cite{fullycommutative}.  For example, $23154$ is fully commutative but not totally commutative.} if there exists $s_{i_1}\dots s_{i_{\ell(w)}}\in {\sf Red}(w)$ with $|i_j-i_k|\geq 2$ for all $j\not=k$.

In their study of ${\sf Red}(w)$, P. Edelman and C. Greene \cite{Edelman.Greene} introduced a family of tableaux. 
Fix a partition $\lambda$ and $w\in S_{\infty}$. We say that $S$ is an \emph{Edelman-Greene tableau} (or \emph{EG tableau}) of type $(\lambda,w)$ if it is a filling of the cells of a Young diagram $\lambda$ such that the cells are strictly increasing on rows and columns, and that if the sequence $i_1,i_2,\dots, i_{|\lambda|}$ results from reading the tableau top-to-bottom and right-to-left, then $s_{i_1}s_{i_2}\dots s_{i_{|\lambda|}}\in {\sf Red}(w)$. Let  ${\sf EG}(\lambda,w)$ be the set of these tableaux. Now,
\begin{equation}
{\sf EG}(w) = \sum_{\lambda} a_{w,\lambda} \text{ , where } a_{w,\lambda}  =|{\sf EG}(\lambda,w)|
\end{equation}
is the \emph{Edelman-Greene statistic} of S.~Billey-B.~Pawlowski \cite{EGStat}.

Define ${\sf inv}(n)$ to be the number of involutions in $S_n$, i.e. the number of permutations $w\in S_n$ such that $w^2$ is the identity permutation.

\begin{theorem}
\label{main}
\begin{equation}
\label{maineq}
\max\{{\sf EG}(w):w\in S_\infty,\ell(w)=n\}  = {\sf inv}(n)
\end{equation}
And the maximum is attained by $w\in S_\infty$ if and only $w$ is totally commutative.

\end{theorem}

We offer three comparisons and contrasts with the literature.

First, B. Pawlowski has proved that ${\mathbb E}[{\sf EG}]\geq (0.072)(1.299)^m$ where the expectation is taken over $w\in S_m$ \cite[Theorem~3.2.7]{PawPhD}). More recently, C. Monical, B. Pankow, and A. Yong show that $EG(w)$ is "typically" exponentially large on $S_m$ \cite[Theorem~1.1]{yongpaper}. In comparison, 
Theorem~\ref{main} combined with a standard estimate for ${\sf inv}(n)$ \cite{Knuth} 
gives 
\begin{equation}
\max\{{\sf EG}(w):w\in S_\infty,\ell(w)=n\} \sim \left(\frac{n}{e}\right)^{\frac{n}{2}}\frac{e^{\sqrt{n}}}{(4e)^{\frac{1}{4}}}
\end{equation}

Second,
in \cite{Pak}, maximums for the Littlewood-Richardson coefficients and their generalization, the Kronecker coefficients, were determined.
We remark that the $a_{w,\lambda}$'s are also generalizations of the Littlewood-Richardson coefficients; this follows from \cite[Corollary~2.4]{BJS}. 

Finally, the results of V.~Reiner-M.~Shimozono \cite{plactification} (see specifically their Theorem~33) appear related to ours. Our work does not depend
on their paper and is combinatorial and self-contained.

\section{Proof of (\ref{maineq})}

Our proof of Theorem~\ref{main} is based on 
a specific relationship between EG tableaux and standard Young tableaux.  Recall that a \emph{standard Young tableau} is a filling of the cells of a Young diagram $\lambda$ with the numbers 1 through $|\lambda|$, each used exactly once, such that the cells are strictly increasing along rows and columns.   The set of standard Young Tableaux of shape $\lambda$ is given by ${\sf SYT}(\lambda)$, and denote $f^\lambda=|{\sf SYT}(\lambda)|$.

Figure \ref{fig:standard} gives several examples of the well-known \emph{standardization} map $\textsf{std}: {\sf SSYT}(\lambda)\to {\sf SYT}(\lambda)$, where ${\sf SSYT}(\lambda)$ is the (countably infinite) set of semistandard tableaux of shape $\lambda$.
Suppose $T\in {\sf SSYT}(\lambda)$ and $k_i$ is the number of $i$'s appearing in $T$. Now replace all $1$'s in $T$ from left to right by $1,2,\ldots,k_1$. Then replace all of the (original) $2$'s in $T$ by $k_1+1,k_1+2,\ldots,k_1+k_2$, \emph{etc}. The result of this procedure is ${\sf std}(T)$. 

If we restrict ${\sf std}$ to the subset of ${\sf SSYT}(\lambda)$ consisting of the (finitely many) tableaux with a given content $\mu$, then it is easy to
see that ${\sf std}$ is an injection. Now, content is not constant on ${\sf EG}(\lambda,w)$. Nevertheless, the conjecture of C.~Monical-B.~Pankow-A.~Yong
\cite[Conjecture~3.12]{yongpaper} is the following:

\begin{figure}
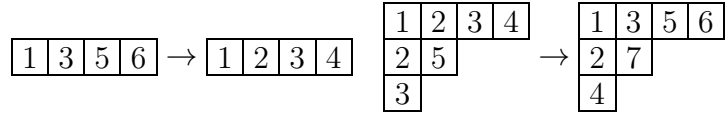

\begin{center}

$\young(1356)\to \young(1234)$\ \   \ \  $\young(1234,25,3) \to \young(1356,27,4)$

\end{center}
\caption{Two examples of standardization.}
\label{fig:standard}
\end{figure}

\begin{theorem}
\label{injection}
The map $\emph{\textsf{std}}: {\sf EG}(\lambda,w)\to {\sf SYT}(\lambda)$ is an injection.
\end{theorem}

\begin{proof}

First, recall that the simple transpositions satisfy:
\begin{equation}
\label{braid1}
s_i s_j=s_js_i \text{ \ for }|i-j|\geq 2
\end{equation}
and
\begin{equation}
\label{braid2}
s_is_{i+1}s_i = s_{i+1}s_is_{i+1}
\end{equation}
where (\ref{braid2}) is the \emph{braid relation}.  Moreover, Tits' Lemma states that any reduced word can be transformed into any other reduced word for the same permutation through a sequence of successive transformations (\ref{braid1}) and (\ref{braid2}).  If $s_{i_1}s_{i_2}\dots s_{i_k}\in {\sf Red}(w)$, define the \emph{support} of $w$ as $\textsf{supp}(w) = \{i_1,i_2,\dots, i_{k}\}$.  

\begin{lemma}
\label{welldefined}
${\sf supp}(w)$ is well-defined.
\end{lemma}
\begin{proof} This follows immediately from Tits' Lemma together with the 
fact that (\ref{braid1}) and (\ref{braid2}) preserve support.
\end{proof}

\begin{lemma}
\label{red order}
For $w\in S_\infty$, if $|a-b|=1$, and there exists a reduced word of $w$ such that all instances of $s_a$ occur before all instances of $s_b$, then the same is true for all reduced words of $w$.
\end{lemma}
\begin{proof}
This holds by Tits' Lemma and examining (\ref{braid1}) and (\ref{braid2}).
\end{proof}

A \emph{descent} of $U\in {\sf SYT}(\lambda)$ is a label $i$ such that $i-1$ is weakly east (and thus strictly north) of $i$. Let $(x,y)$ be the matrix coordinates of a cell in $U$.  Denote the label of cell $(x,y)$ in $U$ by ${\sf Label}_U(x,y)$.  Let the \emph{sweep map} of $U$, ${\sf sweep}(U)$ be the Young tableau of shape $\lambda$, and 
\begin{equation}
{\sf Label}_{{\sf sweep}(U)}(x,y) = |\{k: 1\leq k \leq {\sf Label}_U(x,y), k\text{ is a descent in } U\}|+1.
\end{equation}

\begin{proposition}
${\sf sweep}$ is a map from ${\sf SYT}(\lambda)$ to ${\sf SSYT}(\lambda)$
\end{proposition}
\begin{proof}
Fix $U\in {\sf SYT}(\lambda)$.  For a given cell $(x,y)$ in $U$, ${\sf Label}_U(x,y)< {\sf Label}_U(x,y+1)$, and so the number of descents less than or equal to ${\sf Label}_U(x,y)$ is at most the number of descents less than or equal to ${\sf Label}_U(x,y+1)$, and so by the definition of the sweep map, ${\sf Label}_{{\sf sweep}(U)}(x,y)\leq {\sf Label}_{{\sf sweep}(U)}(x,y+1)$.

Additionally, ${\sf Label}_U(x,y)< {\sf Label}_U(x+1,y)$.  If none of ${\sf Label}_U(x,y)+1, {\sf Label}_U(x,y)+2, \dots, {\sf Label}_U(x+1,y)$ were descents, then each of those labels would be weakly northeast of the one before it, so ${\sf Label}_U(x+1,y)$ would be weakly northeast of ${\sf Label}_U(x,y)$.  This contradicts the fact that $(x+1,y)$ is below $(x,y)$.  Therefore, one of ${\sf Label}_U(x,y)+1, {\sf Label}_U(x,y)+2, \dots, {\sf Label}_U(x+1,y)$ is a descent, and so by the definition of the sweep map, ${\sf Label}_{{\sf sweep}(U)}(x,y)< {\sf Label}_{{\sf sweep}(U)}(x+1,y)$.

Thus we have shown that ${\sf sweep}(U)$ is weakly increasing on rows and strictly increasing on columns, so it is a semistandard Young tableau of shape $\lambda$, and we are done.
\end{proof}

In addition, the $i^{th}$ \emph{sweep} of $U$ is 
\begin{equation}
{\sf sweep}_i(U):= \{(x,y)\in \lambda: {\sf Label}_{{\sf sweep}(U)}(x,y) = i\}.
\end{equation}

\begin{figure}
\begin{center}
$S = \young(1257,257,4,8)$, $T = \young(1357,357,5,8)$, $\textsf{std}(S) = \textsf{std}(T) = \young(1368,257,4,9)$
${\sf sweep}(\textsf{std}(S))  = \young(1234,234,3,5)$
\end{center}
\caption{An example of two semistandard Young tableaux and their images under ${\sf std}$ and ${\sf sweep}$.  Note that $S$ and $T$ are not $EG$ tableaux.}
\label{example}
\end{figure}

\begin{lemma}
\label{leftright}
If $U\in {\sf EG}(\lambda,w)$ and $(x,y),(c,d)\in {\sf sweep}_i({\sf std}(U))$ with $y<d$ for some $i$, then ${\sf Label}_U(x,y)\leq {\sf Label}_U(c,d)$.

\end{lemma}
\begin{proof}
Since $(x,y)$ and $(c,d)$ lie in the same sweep of ${\sf std}(U)$, and $(c,d)$ is to the right of $(x,y)$, the definition of sweep says that ${\sf Label}_{{\sf std}(U)}(x,y)< {\sf Label}_{{\sf std}(U)}(c,d)$.  Therefore, by the definition of standardization, ${\sf Label}_U(x,y)\leq {\sf Label}_U(c,d)$.
\end{proof}

\begin{lemma}
\label{samesweep}
If $U\in {\sf EG}(\lambda,w)$ and ${\sf Label}_{U}(x,y)={\sf Label}_{U}(c,d)$, then ${\sf Label}_{{\sf sweep}({\sf std}(U))}(x,y)={\sf Label}_{{\sf sweep}({\sf std}(U))}(c,d)$.

\end{lemma}
\begin{proof}
Without loss of generality, assume that $(x,y)$ is strictly northeast of $(c,d)$.  This means that ${\sf Label}_{{\sf std}(U)}(x,y)>{\sf Label}_{{\sf std}(U)}(c,d)$.  None of ${\sf Label}_{{\sf std}(U)}(c,d)+1$, ${\sf Label}_{{\sf std}(U)}(c,d)+2$, $\dots$, ${\sf Label}_{{\sf std}(U)}(x,y)$ will be descents, and so $(x,y)$ and $(c,d)$ will be in the same sweep of ${\sf std}(U)$.
\end{proof}

\begin{lemma}
\label{diffsweep}
If $U\in {\sf EG}(\lambda,w)$ and $i<j$, then for $(x,y)\in {\sf sweep}_i({\sf std}(U))$, $(c,d)\in {\sf sweep}_j({\sf std}(U))$, then ${\sf Label}_{U}(x,y) < {\sf Label}_{U}(c,d)$.

\end{lemma}
\begin{proof}
Since ${\sf Label}_{{\sf sweep}({\sf std}(U))}(x,y) =i < j= {\sf Label}_{{\sf sweep}({\sf std}(U))}(c,d)$, it follows from the definition of the sweep map that ${\sf Label}_{{\sf std}(U)}(x,y) < {\sf Label}_{{\sf std}(U)}(c,d)$.  Hence, by the definition of standardization, ${\sf Label}_{U}(x,y) \leq {\sf Label}_{U}(c,d)$.  However, by the contrapositive of Lemma \ref{samesweep}, ${\sf Label}_{U}(x,y) \not= {\sf Label}_{U}(c,d)$, and we are done.
\end{proof}

Now in order to reach a contradiction, assume that there exists $S,T\in {\sf EG}(\lambda,w)$ such that $S\not= T$ and $\textsf{std}(S)=\textsf{std}(T)$.  Since $S\not= T$,
\begin{equation}
D:= \{ (x,y): {\sf Label}_S(x,y)\not= {\sf Label}_T(x,y) \}
\end{equation}
is non-empty.  Define $L = \max\{ i: {\sf sweep}_i({\sf std}(S))\cap D \not= \emptyset \}$.  Let 
\begin{equation}
\label{abdef}
a := \max\{ {\sf Label}_S(x,y): (x,y)\in D\}, \text{ and } 
b := \max\{ {\sf Label}_T(x,y): (x,y)\in D\}.
\end{equation}

There are two cases to consider: either $a=b$ or $a\not= b$.  For the first case, by definition there exists $(x,y),(c,d)\in D$ be such that ${\sf Label}_S(x,y)=a$, and ${\sf Label}_T(c,d)=b$.  By the definition of $D$, ${\sf Label}_T(x,y)\not= {\sf Label}_S(x,y)={\sf Label}_T(c,d)$.  Also, by the definition of $b$, ${\sf Label}_T(c,d) = b \geq {\sf Label}_T(x,y)$, and so ${\sf Label}_T(c,d) > {\sf Label}_T(x,y)$.  By the definition of standardization, this means that ${\sf Label}_{{\sf std}(T)}(c,d) > {\sf Label}_{{\sf std}(T)}(x,y)$.  However, similarly, ${\sf Label}_S(x,y) > {\sf Label}_S(c,d)$, which means that ${\sf Label}_{{\sf std}(S)}(x,y) > {\sf Label}_{{\sf std}(S)}(c,d)$.  However, this contradicts the fact that ${\sf std}(S) = {\sf std}(T)$, completing the proof in this case.

For the second case ($a\not= b$), assume without loss of generality that $b> a$.  By Lemma \ref{welldefined}, some cell in $S$ is labeled $b$ as well, so define
\begin{equation}
B = \{(x,y):{\sf Label}_S(x,y) = b\} \text{ and } C = \min\{y:(x,y)\in B\}.
\end{equation}

\begin{claim}
\label{shareb}
All cells labeled $b$ in $S$ are also labeled $b$ in $T$, and there exists at least one cell labeled $b$ in $T$ that is to the left of column $C$.
\end{claim}
\begin{proof}
Since $b>a$, $B\cap D = \emptyset$, and so if $(c,d)\in B$, ${\sf Label}_T(c,d) = b$ as well. In addition, by the definition of $b$ there exists some cell $(x,y)\in D$ such that ${\sf Label}_T(x,y) = b$, so $(x,y)\not\in B$.  By Lemma \ref{samesweep}, these cells must all be in the same sweep of ${\sf std}(T)$.  We also know that, since $(x,y)\in D$, ${\sf Label}_S(x,y)\leq a < b$, so by Lemma \ref{leftright}, $(x,y)$ must lie to the left of all cells in $B$, and so it must lie to the left of the column with index $C$, completing the proof.
\end{proof}

\begin{claim}
\label{SweepL}
In $T$, all cells labeled $b$ are in ${\sf sweep}_L({\sf std}(T))$, and all cells labeled $a,a+1,\dots,b$ in $S$ are in ${\sf sweep}_L({\sf std}(S))$.
\end{claim}
\begin{proof}
By the definition of $L$, there is some cell $(x,y)\in {\sf sweep}_L({\sf std}(T))\cap D$.  By the definition of $b$, there exists some cell $(c,d)\in D$ such that ${\sf Label}_T(c,d) =b$ and $b\geq {\sf Label}_T(x,y)$, so by the contrapositive of Lemma \ref{diffsweep}, ${\sf Label}_{{\sf sweep}({\sf std}(T))}(c,d)\geq {\sf Label}_{{\sf sweep}({\sf std}(T))}(x,y) = L$.  However, since $(c,d)\in D$, ${\sf Label}_{{\sf sweep}({\sf std}(T))}(c,d)\leq L$, and so ${\sf Label}_{{\sf sweep}({\sf std}(T))}(c,d)= L$.  As a result, since ${\sf Label}_{T}(c,d)= b$, Lemma \ref{samesweep} implies that all cells labeled $b$ in $T$ must be in ${\sf sweep}_L({\sf std}(T))$.

By the argument of the previous paragraph (replacing $T$ with $S$ and $b$ with $a$), all cells labeled $a$ in $S$ must be in ${\sf sweep}_L({\sf std}(S))$.  By Claim \ref{shareb}, any cells in $B$ are labeled $b$ in $T$ as well.  Therefore, since all cells labeled $b$ in $T$ are in ${\sf sweep}_L({\sf std}(T))$, all cells in $B$ are also in ${\sf sweep}_L({\sf std}(T))= {\sf sweep}_L({\sf std}(S))$.  Additionally, the contrapositive of Lemma \ref{diffsweep} implies that any cell labels between $a$ and $b$ in $S$ must occur in ${\sf sweep}_L({\sf std}(S))$ as well, completing the proof.
\end{proof}

For $U\in {\sf EG}(\lambda,w)$, let the \emph{reading word} of $U$, denoted ${\sf Red}(U)$, be $s_{i_1}s_{i_2}\dots s_{i_{|\lambda|}}$, where $i_1,i_2,\dots, i_{|\lambda|}$ is the sequence of labels of $U$ reading from top-to-bottom and right-to-left.  By definition, ${\sf Red}(U)\in {\sf Red}(w)$.

\begin{claim}
\label{columnC}
In all columns with index at least $C$, no cell can be labeled $b-1$ in either $S$ or $T$.
\end{claim}
\begin{proof}
$a\leq b-1<b$, so Claim \ref{SweepL} says that all cells labeled $b-1$ or $b$ in $S$ are in ${\sf sweep}_L({\sf std}(S))$.  By Lemma \ref{leftright} all cells labeled $b-1$ in $S$ must occur strictly to the left of all cells labeled $b$ in $S$, which means none of them can be in a column with index at least $C$.

As a result, all $s_{b-1}$'s will occur after all $s_b$'s in ${\sf Red}(S)$, and so by Lemma \ref{red order}, the same is true for ${\sf Red}(T)$, since we assumed that ${\sf Red}(S),{\sf Red}(T)\in {\sf Red}(w)$.  This means that all cells labeled $b-1$ in $T$ must occur in some column weakly to the left of the leftmost occurrence of a cell labeled $b$ in $T$.  By Claim \ref{shareb}, this is strictly to the left of the column indexed $C$.  Therefore, in all columns with index at least $C$, no cell can be labeled $b-1$ in either $S$ or $T$, so the claim is true.
\end{proof}

Define $G = (\bigcup_{i=L}^\infty {\sf sweep}_i({\sf std}(S)))\cap \{(x,y):y\geq C\}$.

\begin{claim}
\label{samebig}
For all $(x,y)\in G$, ${\sf Label}_S(x,y) = {\sf Label}_T(x,y) \geq b$.
\end{claim}
\begin{proof}
Since Claim \ref{SweepL} says that there is a cell in ${\sf sweep}_L({\sf std}(S))$ labeled $b$ in $S$, every cell in $\bigcup_{i=L+1}^\infty {\sf sweep}_i({\sf std}(S))$ will have a label larger than $b$ in $S$ by Claim \ref{diffsweep}.  The definition of $C$ says that all cells in ${\sf sweep}_L({\sf std}(S))$ in a column labeled at least $C$ will have a label of $b$ or more in $S$.  As a result, all $(x,y)\in G$ have ${\sf Label}_S(x,y) \geq b$.  Since $b> a$, none of these cells are in $D$, and so they have the same labels in $T$ as well, completing the proof.
\end{proof}

Let $s_{i_1}s_{i_2}\dots s_{i_{|\lambda|}} = {\sf Red}(S)$ and let $s_{j_1}s_{j_2}\dots s_{j_{|\lambda|}} = {\sf Red}(T)$.  Let $I$ be the set of all indexes $k$ such that $s_{i_k}$ corresponds to a cell $(x,y)\in G$, and let $M = \max(I)$.  By Claim \ref{samebig}, $i_k=j_k \geq b$ for all $k\in I$.  By Claim \ref{columnC}, $i_a, j_a<b-1$ for $a\leq M$, $a\not\in I$, so $s_{i_k}$ commutes with $s_{i_a}$ for all such $a\leq M$, $a\not\in I$ and $k\in I$.  Therefore, 
\begin{equation}
\prod_{k\in I}s_{i_k} \prod_{a\not\in I}s_{i_a} = {\sf Red}(S) = {\sf Red}(T) = \prod_{k\in I}s_{j_k} \prod_{a\not\in I} s_{j_a}
\end{equation}
and so multiplying both sides by $(\prod_{k\in I}s_{i_k})^{-1}$ results in $\prod_{a\not\in I}s_{i_a} = \prod_{a\not\in I} s_{j_a}$, and we denote the two sides ${\sf Red}(S')$ and ${\sf Red}(T')$ respectively.  However, the definition of $C$ says that $b\not\in {\sf supp}({\sf Red}(S'))$, but Claim \ref{shareb} says that $b\in {\sf supp}({\sf Red}(T'))$.  This contradicts Lemma \ref{welldefined}, and we are done.
\end{proof}

To illustrate the above argument, in Figure \ref{example}, the squares corresponding to the fourth and fifth sweeps in $S$ and $T$ are the same, but not for the third sweep, so in this case, $L=3$, $b = 5$, and $C = 2$.

This means that, by the fact that $S,T\in EG(\lambda,w)$,
\begin{equation}
w= s_7s_5s_7s_2s_5s_1s_2s_4s_8 = s_7s_5s_7s_3s_5s_1s_3s_5s_8
\end{equation}
and by $(\ref{braid1})$, this can be rewritten this as
\begin{equation}
w= s_7s_5s_7s_5s_2s_1s_2s_4s_8 = s_7s_5s_7s_5s_3s_1s_3s_5s_8
\end{equation}
and multiplying both permutations by $s_5s_7s_5s_7$ on the left results is \begin{equation}
s_2s_1s_2s_4s_8 = s_3s_1s_3s_5s_8
\end{equation}
However, only one of the two permutations has $s_5$ in it, contradicting Lemma \ref{welldefined}, and completing the proof.

\begin{corollary}\label{cor1}
\begin{equation}
a_{w,\lambda} \leq f^\lambda
\end{equation}
\end{corollary}

\begin{proof}

This is immediate from Theorem~\ref{injection}. 
\end{proof}

By Corollary~\ref{cor1},
\begin{equation}
{\sf EG}(w) = \sum_{|\lambda|= \ell(w)} a_{w,\lambda} \leq \sum_{|\lambda|= \ell(w)} f^\lambda.
\end{equation}
Taking the maximum over all $w\in S_\infty$ with $\ell(w)=n$ gives 
\begin{equation}
\max\{{\sf EG}(w):w\in S_\infty,\ell(w)=n\}  \leq \sum_{|\lambda|= n} f^\lambda = {\sf inv}(n)
\end{equation}
where the last equality is a consequence of the Schensted correspondence (for example, in \cite{ECII}, Corollary 7.13.9).

For the other direction of $(\ref{maineq})$, fix $n\in \mathbb{N}$ and consider $w_n=s_1s_3\dots s_{2n-1}$.  By inspection, $\ell(w_n) = n$ and any reordering of $s_1s_3\dots s_{2n-1}$ is also a valid reduced word for $w_n$.  Therefore, for each $S\in \textsf{SYT}(\lambda)$ with $|\lambda|=n$, replacing each cell's label $i$ with $2i-1$ is a bijection from $\textsf{SYT}(\lambda)$ to $\textsf{EG}(\lambda,w_n)$, so 
\begin{equation}
f^\lambda = |\textsf{SYT}(\lambda)| = |\textsf{EG}(\lambda,w_n)| = a_{w_n,\lambda}
\end{equation}
and as a result, 
\begin{equation}
{\sf EG}(w_n) = \sum_{|\lambda|=n} a_{w_n,\lambda} = \sum_{|\lambda|=n} f^\lambda = {\sf inv}(n)
\end{equation}
which proves that $(\ref{maineq})$ is an equality.

\section{Classification of the Maximizers of (\ref{maineq})}

For a Young diagram $\lambda$, define $w\in S_\infty$ to be \emph{$\lambda$-maximal} if $a_{w,\lambda} = f^\lambda$. We now classify which $w$ are $\lambda$-maximal for each fixed $\lambda$.

\begin{theorem}
\label{maximal}
Let $\lambda$ be a Young diagram, and let $w$ be a permutation.
\begin{enumerate}
\item
If $\lambda$ has only one row, $w$ is $\lambda$-maximal if and only if there exists $i_1<i_2<\dots < i_{|\lambda|}$ such that $w = s_{i_{|\lambda|}}s_{i_{|\lambda|-1}}\dots s_{i_1}$.
\item
If $\lambda$ has only one column, $w$ is $\lambda$-maximal if and only if $w = s_{i_1}s_{i_2}\dots s_{i_{|\lambda|}}$ for some $i_1<i_2<\dots < i_{|\lambda|}$.
\item
If $\lambda$ has more than one row and more than one column, $w$ is $\lambda$-maximal if and only if $\ell(w) = |\lambda|$ and $w$ is totally commutative.
\end{enumerate}
\end{theorem}
\begin{proof}
(1) and (2) are trivial by the definition of $\lambda$-maximal.  Therefore, the rest of this proof is devoted to proving (3).  The reverse direction follows from the following lemma:

\begin{lemma}
\label{totally commutative}
If $w$ is totally commutative, then it is $\lambda$-maximal for all $|\lambda|=\ell(w)$.
\end{lemma}
\begin{proof}
Let $i_1,\dots i_k$ be as in the definition of totally commutative.  Then by (\ref{braid1}), 
\begin{equation}
s_{i_{\sigma(1)}}s_{i_{\sigma(2)}}\dots s_{i_{\sigma(\ell(w))}} \in {\sf Red}(w) \text{ for all } \sigma\in S_{\ell(w)}.
\end{equation}
For any $T \in {\sf SYT}(\lambda)$, replacing the label $k$ with the $k^{th}$ smallest element of ${\sf supp}(w)$ turns $T$ into an element $T'\in {\sf EG}(w,\lambda)$.  This mapping $T\mapsto T'$ is clearly an injection, so this and Corollary \ref{cor1} combine to say that $a_{w,\lambda}=f^\lambda$.
\end{proof}

The forward direction also requires a lemma.

\begin{lemma}
\label{non-maximal}
If $|{\sf supp}(w)| < \ell(w)$, then $w$ is not $\lambda$-maximal for any $\lambda$.
\end{lemma}
\begin{proof}
Assume for the sake of contradiction that there exists some Young diagram $\lambda$ such that $w$ is $\lambda$-maximal.  Fix an arbitrary $U\in {\sf EG}(w,\lambda)$.  Since $|{\sf supp}(w)| < \ell(w) = |\lambda|$, there exists $(x,y)$ and $(c,d)$ such that ${\sf Label}_U(x,y)={\sf Label}_U(c,d)$.  Since $U$ is strictly increasing on rows and columns, without loss of generality $(c,d)$ is strictly northeast of $(x,y)$, and in particular $\lambda$ must have more than one row and more than one column.  As a result, ${\sf Label}_{{\sf std}(U)}(x,y) < {\sf Label}_{{\sf std}(U)}(c,d)$.  This is a contradiction, as then no element of ${\sf EG}(w,\lambda)$ maps to $S\in {\sf SYT}(\lambda)$, the unique element of ${\sf SYT}(\lambda)$ where cells are labeled $1$ through $|\lambda|$ by going from left to right and top to bottom, but ${\sf std}:{\sf EG}(w,\lambda)\to {\sf SYT}(\lambda)$ is an injection between two equally sized finite sets by Theorem \ref{injection}, so it should be a surjection. 
\end{proof}

Let $\lambda$ have more than one row and more than one column, and assume that $w$ is $\lambda$-maximal.  By definition, $\ell(w)=|\lambda|$, and since ${\sf std}:{\sf EG}(w,\lambda)\to {\sf SYT}(\lambda)$ is an injection between two finite sets of the same size by Theorem \ref{injection}, it is a bijection.  

By Lemma \ref{non-maximal}, ${\sf supp}(w) = \{i_1,i_2,\dots, i_{\ell(w)}\}$, where we can say $i_1<i_2<\dots<i_{\ell(w)}$ without loss of generality.  Therefore, ${\sf std}$ maps the label $i_k$ to $k$ and ${\sf std}^{-1}$ maps the label $k$ to $i_k$ for each $k$.  Now assume for the sake of contradiction that $w$ is not totally commutative.  This means that  $m := \min \{j:i_j+1 = i_{j+1}\}$ is finite.

Recall the definition of ${\sf Red}(U)$ for $U\in {\sf EG}(w,\lambda)$ from directly before Claim \ref{columnC}.

\begin{claim}
\label{sytorder}
If $T \in {\sf SYT}(\lambda)$ is such that $s_{i_m}$ occurs before $s_{i_{m+1}}$ in ${\sf Red}({\sf std}^{-1}(T))$, then $s_{i_{m}}$ occurs before $s_{i_{m+1}}$ in ${\sf Red}({\sf std}^{-1}(T'))$ for all other $T' \in {\sf SYT}(\lambda)$.
\end{claim}
\begin{proof}
This follows immediately from Lemma \ref{red order} and the fact that each simple transposition occurs at most once in each element of ${\sf Red}(w)$.
\end{proof}

There are three cases to consider: $m=1$, $m=|\lambda|-1$ and $\lambda$ is a rectangle, and the case where neither of the above is true.

\noindent
{\sf Case 1:} ($m=1$)  Let $T,T'\in {\sf SYT}(\lambda)$ be such that ${\sf Label}_T(2,1)=2$ and ${\sf Label}_{T'}(1,2)=2$.  As a result, because ${\sf Label}_T(1,1)= {\sf Label}_{T'}(1,1)=1$, $s_{1}$ occurs before $s_2$ in ${\sf Red}({\sf std}^{-1}(T))$, but $s_{2}$ occurs before $s_1$ in ${\sf Red}({\sf std}^{-1}(T'))$.  This contradicts Claim \ref{sytorder}.

\noindent
{\sf Case 2:} ($m=|\lambda|-1$ and $\lambda$ is rectangular) Say that $\lambda$ is a $a\times b$ rectangle so that $m = ab-1$.  Let $T,T'\in {\sf SYT}(\lambda)$ be such that ${\sf Label}_T(a-1,b)=m$ and ${\sf Label}_{T'}(a,b-1)=m$.  As a result, because ${\sf Label}_T(a,b)= {\sf Label}_{T'}(a,b)=m+1$, $s_{m}$ occurs before $s_{m+1}$ in ${\sf Red}({\sf std}^{-1}(T))$, but $s_{m+1}$ occurs before $s_m$ in ${\sf Red}({\sf std}^{-1}(T'))$.  This once again contradicts Claim \ref{sytorder}.

\noindent
{\sf Case 3:} (Neither {\sf Case 1} nor {\sf Case 2}) There exists some $T\in {\sf std}(\lambda)$ such that the cell labeled $m$ in $T$ (denoted $(a,b)$) is strictly northeast of the cell $m+1$ in $T$ (denoted $(c,d)$).  From this, let $T'\in {\sf std}(\lambda)$ be identical to $T$ except that ${\sf Label}_{T'}(a,b) = m+1$ and ${\sf Label}_{T'}(c,d) = m$.  As before, $s_{m}$ occurs before $s_{m+1}$ in ${\sf Red}({\sf std}^{-1}(T))$, but $s_{m+1}$ occurs before $s_m$ in ${\sf Red}({\sf std}^{-1}(T'))$, contradicting Claim \ref{sytorder}.

This completes the proof. 
\end{proof}

The above theorem allows us to characterize the permutations that maximize the Edelman-Greene statistic.

\begin{corollary}
${\sf EG}(w)={\sf inv}(\ell(w))$ if and only if $w$ is totally commutative.

\end{corollary}

\begin{proof}
The reverse direction follows from Lemma \ref{totally commutative}.  For the forward direction, consider three cases, based on the size of $\ell(w)$.  If $\ell(w)=1$, then $w$ is always totally commutative.  If $\ell(w)=2$, then Lemma \ref{non-maximal} says that ${\sf supp}(w) = \{i_1,i_2\}$, and $s_{i_1}s_{i_2} = s_{i_2}s_{i_1}$, so they commute and so $w$ is totally commutative.  For $\ell(w)\geq 3$, there exists some $\lambda$ with $|\lambda|=\ell(w)$ and $\lambda$ having at least two rows and at least two columns.  Since ${\sf EG}(w)={\sf inv}(\ell(w))$, $w$ must be $\lambda$-maximal, and so by Theorem \ref{maximal}, $w$ must be totally commutative.
\end{proof}

\section*{Acknowledgment}

We thank Brendan Pawlowski and Alexander Yong for helpful discussions on the subject.


\begin{thebibliography}{9}

\bibitem{BJS} S.~Billey, W.~Jockusch and R.~P.~Stanley, \emph{Some combinatorial properties of Schubert polynomials},
J.~Algebraic Combin.~{\bf 2}(1993), no.~4, 345--374.

\bibitem{EGStat}
 S. Billey and B. Pawlowski, \textit{Permutation patterns, Stanley symmetric functions, and generalized Specht modules}, J. Combin. Theory Ser. A 127 (2014), 85–120.

 
 \bibitem{Edelman.Greene} P.~Edelman and C.~Greene,
\emph{Balanced tableaux}, Adv. in Math. 63 (1987), no. 1, 42--99.

\bibitem{Knuth} D. Knuth, The art of computer programming. Volume 3. Sorting and searching. Addison-Wesley Series in Computer Science and Information Processing. Addison-Wesley Publishing Co., Reading, Mass.-London-Don Mills, Ont., 1973. {\rm xi}+722 pp.

\bibitem{yongpaper}
  C. Monical, B. Pankow, A. Yong, \textit{Reduced word enumeration, complexity, and randomization}, preprint, 2019. arXiv:1901.03247 



\bibitem{Pak} I. Pak, G. Panova and D. Yeliussizov, \emph{On the largest Kronecker and Littlewood}-\emph{Richardson coefficients},
arXiv:1804.04693
 
  \bibitem{PawPhD} B. Pawlowski, \emph{Permutation diagrams in symmetric function theory and Schubert calculus}, PhD thesis, University of Washington, 2014

\bibitem{plactification} V. Reiner and M. Shimozono, \emph{Plactification}, J. Algebraic Combin. 4:4 (1995), 331–351.


 \bibitem{ECII} R.~P.~Stanley, \emph{Enumerative combinatorics}, Vol. 2. With a foreword by Gian-Carlo Rota and appendix 1 by Sergey Fomin. Cambridge Studies in Advanced Mathematics, 62. Cambridge University Press, Cambridge, 1999. xii+581 pp.
 
\bibitem{fullycommutative} J.R. Stembridge. \emph{On the Fully Commutative Elements of Coxeter Groups} Journal of Algebraic Combinatorics (1996) 5: 353.
 

\end{thebibliography}
\end{document}